\documentclass[12pt,reqno]{article}

\usepackage[usenames]{color}
\usepackage{amsmath}
\usepackage{amsthm}
\usepackage{amssymb}
\usepackage{amsfonts}
\usepackage{amscd}
\usepackage{mathtools}
\usepackage{ytableau}
\usepackage{subcaption}
\usepackage{etoolbox}
\makeatletter 
\patchcmd\caption@subtypehook{\let\label\subcaption@label}
{\let\label\subcaption@label\let\ltx@label\subcaption@label}{}{\fail}
\makeatother
\DeclareRobustCommand{\gaussian}{\genfrac[]{0pt}{}}
\DeclareMathOperator{\inv}{inv}
\DeclareMathOperator{\maj}{maj}
\DeclareMathOperator{\comaj}{comaj}
\DeclareMathOperator{\des}{des}
\DeclareMathOperator{\imaj}{imaj}
\DeclareMathOperator{\icomaj}{icomaj}
\DeclareMathOperator{\ides}{ides}
\DeclareMathOperator{\st}{st}
\DeclareMathOperator{\qExp}{E}
\DeclareMathOperator{\qexp}{e}

\usepackage[colorlinks=true,
linkcolor=webgreen,
filecolor=webbrown,
citecolor=webgreen]{hyperref}

\definecolor{webgreen}{rgb}{0,.5,0}
\definecolor{webbrown}{rgb}{.6,0,0}

\usepackage{color}
\usepackage{fullpage}
\usepackage{float}

\usepackage{latexsym}
\usepackage{epsf}
\usepackage{breakurl}

\newcommand{\seqnum}[1]{\href{https://oeis.org/#1}{\rm \underline{#1}}}

\begin{document}
	\theoremstyle{plain}
	\newtheorem{theorem}{Theorem}
	
	\newtheorem{corollary}[theorem]{Corollary}
	\newtheorem{lemma}[theorem]{Lemma}
	\newtheorem{proposition}[theorem]{Proposition}
	
	\theoremstyle{definition}
	\newtheorem{definition}[theorem]{Definition}
	\newtheorem{example}[theorem]{Example}
	\newtheorem{conjecture}[theorem]{Conjecture}
	
	\theoremstyle{remark}
	\newtheorem{remark}[theorem]{Remark}
	\begin{center}
		\vskip 1cm{\LARGE\bf {Distributions of Inversions and Descents over Integer Compositions}
		}
		\vskip 1cm
		\large
		Eder G. Santos\\
		Brazil \\
		\href{mailto:email}{\tt ederguidsa@gmail.com} \\
	\end{center}
	
	\vskip .2 in
	\begin{abstract}
		We derive a generating function for the number of integer compositions of $n$ into $k$ parts (i.e., $k$-compositions of $n$) with a given number of inversions, and obtain similar results for $k$-compositions of $n$ with a given number of descents. Our approach relies on a known bijection that associates each integer composition $\sigma$ with a pair $(\pi,\lambda)$, where $\pi$ is a permutation and $\lambda$ is an integer partition. We show that the distribution of inversions and the distribution of descents over $k$-compositions are related, respectively, to the distribution of (maj,inv) and to the distribution of (inv,des) over permutations of $\{1,2,\ldots,k\}$, where maj, inv, and des denote the classical permutation statistics major index, inversion number, and descent number, respectively.
	\end{abstract}
	
	\section{Introduction}
	A \emph{composition} $\sigma$ of the integer $n$ into exactly $k$ parts (or a $k$-composition of $n$) is an ordered $k$-tuple of positive integers $(\sigma_1,\sigma_2,\ldots,\sigma_k)$ such that $\sigma_1+\sigma_2+\cdots+\sigma_k=n$. The set containing all $k$-compositions, for unrestricted $n$, is denoted by $\mathbf{C}_k$. Similarly, a \emph{partition} $\lambda$ of the integer $n$ into exactly $k$ parts (or a $k$-partition of $n$) is a $k$-tuple of weakly decreasing positive integers $(\lambda_1,\lambda_2,\ldots,\lambda_k)$ such that $\lambda_1+\lambda_2+\cdots+\lambda_k=n$. If $\lambda$ is a partition of $n$ we may also write $\lambda\vdash n$. The \emph{length} of a partition $\lambda$, denoted by $l(\lambda)$, is its total number of parts and the \emph{multiplicity} of part $i$ in $\lambda$, denoted by $m_i^\lambda$, is the number of times a part $i$ appears in $\lambda$. A permutation $\pi=\pi_1\pi_2\ldots\pi_k$ is a rearrangement of the elements of the set $[\mathbf{k}]=\{1,2,\ldots,k\}$. We define $\pi^{-1}_i$ as the index $j$ such that $\pi_j = i$, and the permutation $\pi^{-1}=\pi^{-1}_1\pi^{-1}_2\cdots\pi^{-1}_k$ as the \emph{inverse permutation} of $\pi$. The set containing all permutations of the elements in $[\mathbf{k}]$ is denoted by $\mathbf{S}_k$. Whenever necessary to unify the notation, we also consider permutations as an ordered $k$-tuple $\pi=(\pi_1,\pi_2,\ldots,\pi_k)=\pi_1\pi_2\ldots\pi_k$.
	
	It is often convenient to represent a $k$-partition $\lambda$ graphically by its \emph{Young diagram}, defined as a left-justified array of square cells with $k$ rows, where the $i$-th row (from top to bottom) contains $\lambda_i$ cells. In this case, we say that the Young diagram $T$ has \emph{shape} $\lambda$. By a slight abuse of notation, we also use $\lambda$ to denote the set of coordinates of the cells in the Young diagram of shape $\lambda$. For a cell $u=(i,j)\in \lambda$, in the $i$-th row and $j$-th column, the \emph{hook length}, $h(u)$, is defined as the number of cells to the right of $u$ in the same row and below $u$ in the same column, including the cell $u$ itself. Figure~\ref{fig:youngD} shows an example of computation of hook lengths in a given shape. A \emph{standard Young tableau} of shape $\lambda$ is a filling of the $n$ cells of the Young diagram of shape $\lambda$ with all the integers in $[\mathbf{n}]$, with no repetition, such that each row and each column form increasing sequences. Figure~\ref{fig:youngT} shows an example of such tableau.
	\begin{figure}[ht]
		\centering 
		\subcaptionbox{\label{fig:youngD}}{
			\begin{ytableau}
				7 & 5 & 3 & 2 \\
				6 & 4 & 2 & 1 \\
				3 & 1\\
				1 
		\end{ytableau}}
		\hspace{10em}
		\subcaptionbox{\label{fig:youngT}}{
			\begin{ytableau}
				1 & 3 & 4 & 8 \\
				2 & 6 & 9 & 11 \\
				5 & 7\\
				10 
		\end{ytableau}}
		\captionsetup{subrefformat=parens}
		\caption{\subref*{fig:youngD}~Young diagram of shape $\lambda=(4,4,2,1)$ filled with the hook lengths of each cell. \subref*{fig:youngT}~Example of standard Young tableau of shape $\lambda=(4,4,2,1)$.}
	\end{figure}
	
	We recall some standard definitions related to permutation statistics that can be applied in the study of compositions as well. Let $\alpha=(\alpha_1,\alpha_2,\ldots,\alpha_k)$ be $k$-tuple of positive integers, with repetitions allowed. An \emph{inversion} of $\alpha$ is a pair of indices $(i,j)$ such that $i<j$ and $\alpha_i>\alpha_j$. A \emph{descent} of $\alpha$ is an index $i$ such that $\alpha_i>\alpha_{i+1}$. The \emph{descent set}, $D(\alpha)$, is the set that contains all descents of $\alpha$ and its number of elements is the \emph{descent number} of $\alpha$, denoted by $\des(\alpha)$. The \emph{inversion number}, $\inv(\alpha)$, is defined as the number of inversions of $\alpha$, while the \emph{major index}, $\maj(\alpha)$, is defined as the sum of all descents of $\alpha$. We also define the \emph{comajor index}, $\comaj(\alpha)$, as the sum of the complements to $k$ of all descents of $\alpha$, i.e., $\comaj(\alpha)=\sum_{i\in D(\alpha)}(k-i)$. The \emph{sum statistic} is $\vert\alpha\vert=\alpha_1+\alpha_2+\cdots+\alpha_k$ and generically denoted by $\vert\cdot\vert$. If $\alpha$ is a permutation, we also define $\imaj(\alpha)=\maj(\alpha^{-1})$, $\icomaj(\alpha)=\comaj(\alpha^{-1})$, and $\ides(\alpha)=\des(\alpha^{-1})$. 
	
	Joint statistics are represented by a tuple in the format $(\st_1,\st_2,\ldots,\st_n)$, where $\st_i$ is one of the aforementioned statistics. The \emph{distribution} of $(\st_1,\st_2,\ldots,\st_n)$ over a given set $\mathbf{A}_k$, containing $k$-tuples of positive integers, is 
	\begin{align*}
		\sum_{\alpha\in\mathbf{A}_k} p_1^{\st_1(\alpha)}p_2^{\st_2(\alpha)}\cdots p_n^{\st_n(\alpha)}.
	\end{align*}
	
	Heubach et al.\ \cite{heubach} introduced several enumeration problems concerning inversions in compositions, establishing the notation of the first four terms listed below. We have complemented the notation with the last two terms in order to add descent related statistics.
	\begin{center}
		\begin{tabular}{cl}
			$\displaystyle ic(n)$ & Number of inversions of all compositions of $n$ (\seqnum{A189052}). \\[2mm]
			$\displaystyle ic(n,k)$ & Number of inversions of all $k$-compositions of $n$ (\seqnum{A189073}). \\[2mm]
			$\displaystyle ic_r(n)$ & Number of compositions of $n$ with $r$ inversions (\seqnum{A189074}).\\[2mm]
			$\displaystyle ic_r(n,k)$ & Number of $k$-compositions of $n$ with $r$ inversions.\\[2mm]
            $\displaystyle dc_r(n)$ & Number of compositions of $n$ with $r$ descents (\seqnum{A238343}, \seqnum{A238344}).\\[2mm]
            $\displaystyle dc_r(n,k)$ & Number of $k$-compositions of $n$ with $r$ descents.
		\end{tabular}
	\end{center}
	
	Heubach et al.\ \cite{heubach} gave explicit expressions for $ic(n)$, $ic(n,k)$, and for $ic_1(n)$. Later, Knopfmacher et al.\ \cite{knopf} analyzed the asymptotic behavior of $ic_r(n)$ and studied $ic_1(n)$ to $ic_5(n)$, constructing the respective generating functions using the inclusion-exclusion principle. Archibald et al.\ \cite{arch} studied $ic_r(n)$, focusing on the parity of $r$.
	
	The purpose of this paper is to investigate $ic_r(n,k)$, $ic_r(n)$, $dc_r(n,k)$, and $dc_r(n)$ for general $n$, $k$ and $r$, deriving explicit expressions for the respective generating functions. Our approach is based on the fact that every composition $\sigma$ can be bijectively mapped to a pair $(\pi,\lambda)$, where $\pi$ is a permutation and $\lambda$ is a partition, with $\sigma$, $\pi$, and $\lambda$ satisfying $\maj(\pi)+\vert\lambda\vert=\vert\sigma\vert$. Since $\lambda$ has a fixed order of parts, we expect that all information about $\inv(\sigma)$, $\des(\sigma)$, and other related statistics is encoded in $\pi$. In fact, in Section~\ref{sec:results}, Theorem~\ref{thm:jointstat}, we prove that the distributions of several statistics on $\mathbf{C}_k$ are directly related to distributions of corresponding statistics on $\mathbf{S}_k$.
	
	This paper is organized as follows: In Section \ref{sec:permstats}, we revisit results available in the literature concerning the distribution two permutation statistics, namely $(\maj,\inv)$ and $(\inv,\des)$. In Section \ref{sec:results}, we derive all results concerning $ic_r(n,k)$, $ic_r(n)$, $dc_r(n,k)$, and $dc_r(n)$. Throughout this paper, we provide links to relevant entries in the {O}n-{L}ine {E}ncyclopedia of {I}nteger {S}equences (OEIS) \cite{OEIS} and use the following notation:
	\begin{center}
		\begin{tabular}{cl}
			$\displaystyle [n]_q$ & $q$-analog of the integer $n$, $\displaystyle\frac{1-q^n}{1-q}=1+q+q^2+\cdots+q^{n-1}$.\\[5mm]
			$\displaystyle [n]_q!$ & $q$-analog of $n!$, $[n]_q[n-1]_q\ldots[2]_q[1]_q$, $[0]_q!=1$.\\[2mm]
			$\displaystyle(q)_n$ & $(1-q)(1-q^2)\cdots (1-q^n)=[n]_q!(1-q)^n$, $(q)_0=1$.\\[2mm]
			$\displaystyle\gaussian{n}{k}_q$ & Gaussian binomial coefficient, $\displaystyle\frac{(q)_n}{(q)_k(q)_{n-k}}$.\\[5mm]
			$\displaystyle{f^{\lambda}}$ & Number of standard Young tableaux of shape $\lambda$ \cite{frame}, $\displaystyle\frac{n!}{\prod_{u\in\lambda}h(u)}$, where $n=\vert\lambda\vert$.\\[5mm]
			$\displaystyle{f^{\lambda}(q)}$ & $q$-analog of $f^\lambda$ \cite{stanley2}, $\displaystyle\frac{q^{b(\lambda)}[n]_q!}{\prod_{u\in\lambda}[h(u)]_q}$, where $n=\vert\lambda\vert$ and $b(\lambda)=\sum_{i}(i-1)\lambda_i$. \\[5mm]
             $\displaystyle g^\lambda(q)$ & $\displaystyle\frac{l(\lambda)![n]_q!}{\prod_{i}m_i^\lambda!([i]_q!)^{m_i^\lambda}}$, where $n=\vert\lambda\vert$.
		\end{tabular}
	\end{center}

    \section{Revisiting two joint statistics on permutations}\label{sec:permstats}
	\subsection{The distribution of \texorpdfstring{$(\maj,\inv)$}{(maj,inv)} over~\texorpdfstring{$\mathbf{S}_k$}{Sk}}
	In the literature, the distribution of $(\maj,\inv)$ over $\mathbf{S}_k$, which we denote by $H_k(p,q)$, is often related to the expansion of the infinite product
	\begin{align}\label{eq:prod}
		\prod_{m,n=0}^{\infty}\frac{1}{1-p^nq^mt}=\sum_{k=0}^{\infty}\frac{H_k(p,q)t^k}{(p)_k(q)_k}.
	\end{align}
	
	Carlitz \cite{carlitz} was the first to demonstrate that $H_k(p,q)$ defined by \eqref{eq:prod} is a polynomial in $p$ and $q$ with integral coefficients. He also derived recurrences to calculate the respective expressions.
	\begin{theorem}[Carlitz \cite{carlitz}]
		Let $H_k(p,q)$ be defined as in \eqref{eq:prod}. Then, $H_k(p,q)$ is a polynomial of degree $\binom{k}{2}$ in both variables with integral coefficients. Additionally, the following recurrence holds:
		\begin{align}
			H_k(p,q)&=\sum_{j=0}^{k-1}p^j\frac{(p)_{k-1}}{(p)_j}\gaussian{k}{j}_qH_j(p,q),\label{eq:carlitzRec}
		\end{align}
		with initial condition $H_0(p,q)=1$.
	\end{theorem}
	
	The connection with permutation statistics was only established later, by Roselle \cite{roselle}, who provided a simple combinatorial interpretation for $H_k(p,q)$, that can be stated as follows \cite[Sol. Ex.\ 27, Sec.\ 5.1.1]{knuth}:
	\begin{theorem}[Roselle \cite{roselle}]
		Let $H_k(p,q)$ be defined as in \eqref{eq:prod}. Then,
		\begin{align*}
			H_k(p,q)=\sum_{\pi\in\mathbf{S}_k}p^{\imaj(\pi)}q^{\maj(\pi)}.
		\end{align*}
	\end{theorem}
	
	Then, Foata and Sch\"utzenberger \cite{fs} proved an equivalence between three pairs of permutation statistics, establishing the interpretation of $H_k(p,q)$ as the distribution of $(\maj,\inv)$ over $\mathbf{S}_k$.
	
	\begin{theorem}[Foata and Sch\"utzenberger \cite{fs}]\label{thm:foata}
		The joint statistics $(\imaj,\maj)$, $(\inv,\imaj)$, and $(\maj,\inv)$ have the same symmetric distribution over $\mathbf{S}_k$, i.e.,
		\begin{align*}
			\sum_{\pi\in\mathbf{S}_k}p^{\imaj(\pi)}q^{\maj(\pi)}=\sum_{\pi\in\mathbf{S}_k}p^{\inv(\pi)}q^{\imaj(\pi)}=\sum_{\pi\in\mathbf{S}_k}p^{\maj(\pi)}q^{\inv(\pi)}.
		\end{align*}
	\end{theorem}
	
	In the study of symmetric functions, the polynomial $H_k(p,q)$ appears as a special case of a result known in the literature as ``Cauchy identity''. The reader may consult MacDonald \cite[Sec.\ I.4]{mac} or Stanley \cite[Thm.\ 7.12.1]{stanley2} for additional references.
	
	\begin{theorem}[Cauchy identity] \label{thm:cauchy}
		Let $\{x_i\mid i\geq1\}$ and $\{y_j\mid j\geq1\}$ be infinite sets of variables, and let $s_{\lambda}$ be the Schur function indexed by the integer partition $\lambda$. Then,
		\begin{align}
			\prod_{i,j=1}^{\infty}\frac{1}{1-x_iy_j}=\sum_{\lambda}s_{\lambda}(x_1,x_2,\ldots)s_{\lambda}(y_1,y_2,\ldots),\label{eq:cauchy}
		\end{align}
		where $\lambda$ runs through all integer partitions.
	\end{theorem}
	Using Theorem~\ref{thm:cauchy}, Stanley \cite[Cor.\ 7.23.9]{stanley2} derived an explicit expression for $H_k(p,q)$:
	
	\begin{theorem}[Stanley \cite{stanley2}]\label{thm:hxy}
		We have
		\begin{align}
			H_k(p,q) = \sum_{\lambda\vdash k}f^{\lambda}(p)f^{\lambda}(q).\label{eq:hxy}
		\end{align}
	\end{theorem}
	\begin{proof}
		We set $x_i = p^{i-1}$ and $y_j=q^{j-1}t$ in \eqref{eq:cauchy} and use the homogeneity of $s_\lambda$ to write
		\begin{align*}
			\sum_{k=0}^{\infty}\frac{H_k(p,q)t^k}{(p)_k(q)_k}=\prod_{m,n=0}^{\infty}\frac{1}{1-p^nq^mt}=\sum_{\lambda}s_{\lambda}(1,p,p^2,\ldots)s_{\lambda}(1,q,q^2,\ldots)t^{\vert\lambda\vert}.
		\end{align*}
		Grouping the sum by $\vert\lambda\vert=k$ and comparing the coefficients of $t^k$, we get
		\begin{align*}
			\frac{H_k(p,q)}{(p)_k(q)_k}=\sum_{\lambda\vdash k}s_{\lambda}(1,p,p^2,\ldots)s_{\lambda}(1,q,q^2,\ldots).
		\end{align*}
		Now, using the principal specialization of $s_\lambda$ \cite[Cor.\ 7.21.3]{stanley2}, given by
		\begin{align*}
			s_{\lambda}(1,q,q^2,\ldots)=\frac{q^{b(\lambda)}}{\prod_{u\in\lambda}(1-q^{h(u)})},
		\end{align*}
		we get
		\begin{align*}
			H_k(p,q)=\sum_{\lambda\vdash k}\frac{p^{b(\lambda)}(p)_kq^{b(\lambda)}(q)_k}{\prod_{u\in\lambda}(1-p^{h(u)})(1-q^{h(u)})},
		\end{align*}
		and \eqref{eq:hxy} follows after trivial simplifications.
	\end{proof}
	It is straightforward to verify that \eqref{eq:hxy} is a $(p,q)$-analog of the well-known identity
	\begin{align*}
		k! = \sum_{\lambda\vdash k}(f^\lambda)^2.
	\end{align*}
	
	The polynomial $f^{\lambda}(q)$ is a $q$-analog of $f^\lambda$, given by the classical hook-length formula of Frame, Robinson, and Thrall \cite[Cor.\ 7.21.6]{frame,stanley2}. Additionally, Stanley \cite[Cor.\ 7.21.5]{stanley2} obtained a combinatorial interpretation for the coefficients of $f^{\lambda}(q)$ by extending the concepts of descent and major index to standard Young tableaux. 
    
    Finally, we list some values of $H_k(p,q)$, computed using \eqref{eq:hxy}:
    \begin{align*}
        H_0(p,q)&=1, \\
        H_1(p,q)&=1, \\
        H_2(p,q)&=1+pq,\\
        H_3(p,q)&=1+(p+p^2)q+(p+p^2)q^2+p^3q^3,\\
        H_4(p,q)&=1+(p+p^2+p^3)q+(p+2p^2+p^3+p^4)q^2+(p+p^2+2p^3+p^4+p^5)q^3+\\
        &+(p^2+p^3+2p^4+p^5)q^4+(p^3+p^4+p^5)q^5+p^6q^6.
    \end{align*}
    
	\subsection{The distribution of \texorpdfstring{$(\inv,\des)$}{(inv,des)} over \texorpdfstring{$\mathbf{S}_k$}{Sk}}\label{sec:invdes}
    In the literature, the distribution of $(\inv,\des)$ over $\mathbf{S}_k$, which we denote by $A_k(q,t)$, is known as the $q$-Eulerian polynomial, which is a $q$-analog of the classical Eulerian polynomial. Stanley \cite{stanley3} proved that the $q$-Eulerian polynomials satisfy the following identity:
    \begin{align}
        \sum_{k=0}^\infty A_k(q,t)\frac{z^k}{[k]_q!}=\frac{1-t}{\qExp_q(z(t-1))-t},\label{eq:geneuler}
    \end{align}
    where $\qExp_q(z)=\sum_{k\geq0}q^{\binom{k}{2}}\frac{z^k}{[k]_q!}$. We expand this generating function to obtain an explicit expression for $A_k(q,t)$ in terms of a partition indexed sum, as done in Theorem~\ref{thm:hxy}.
    
    \begin{theorem}
        We have $A_0(q,t)=1$ and, for $k\geq1$,
        \begin{align}
            A_k(q,t)=\sum_{\lambda\vdash k}t^{l(\lambda)-1}(1-t)^{k-l(\lambda)}g^\lambda(q).\label{eq:qeuler}
        \end{align}
    \end{theorem}
    \begin{proof}
    Let $\qexp_q(z)=\sum_{k\geq 0}\frac{z^k}{[k]_q!}$. Using \eqref{eq:geneuler} and the well-known identity $\qexp_q(-z)\qExp_q(z)=1$, it is possible to obtain 
    \begin{align}
        1+t\sum_{k\geq1}A_k(q,t)\frac{z^k}{[k]_q!}=\left(1-t\sum_{k\geq 1}(1-t)^{k-1}\frac{z^k}{[k]_q!}\right)^{-1}.
    \end{align}
    Expanding the right-hand side as a geometric series and simplifying common terms, we get
    \begin{align*}
        \sum_{k\geq1}A_k(q,t)\frac{z^k}{[k]_q!}&=\sum_{m\geq1}t^{m-1}\left(\sum_{k\geq 1}(1-t)^{k-1}\frac{z^k}{[k]_q!}\right)^{m}=\\
        &=\sum_{m\geq1}t^{m-1}\sum_{k\geq 1}(1-t)^{k-m}\sum_{\substack{k_1+k_2+\cdots+k_m=k \\ k_i \ge 1}}\frac{[k]_q!}{[k_1]_q![k_2]_q!\ldots[k_m]_q!}\frac{z^k}{[k]_q!}
    \end{align*}
    Now, we rearrange summation indices and note that contributions from all $m$-compositions of $k$ corresponding to a fixed $m$-partition are identical, therefore
    \begin{align*}
        \sum_{k\geq1}A_k(q,t)\frac{z^k}{[k]_q!}&=\sum_{k\geq1}\sum_{m= 1}^kt^{m-1}(1-t)^{k-m}\sum_{\substack{\lambda\vdash k\\l(\lambda)=m}}\frac{m![k]_q!}{\prod_im_i^\lambda!([i]_q!)^{m_i^\lambda}}\frac{z^k}{[k]_q!}=\\
        &=\sum_{k\geq1}\sum_{\lambda\vdash k}t^{l(\lambda)-1}(1-t)^{k-l(\lambda)}g^\lambda(q)\frac{z^k}{[k]_q!}.
    \end{align*}
    Comparing coefficients of $\frac{z^k}{[k]_q!}$ on both sides yields the result.
    \end{proof}
    Finally, we list some values of $A_k(q,t)$, computed using \eqref{eq:qeuler}:
    \begin{align*}
        A_1(q,t)&=1,\\
        A_2(q,t)&=1+qt,\\
        A_3(q,t)&=1 + (2q + 2q^2)t + q^3t^2,\\
        A_4(q,t)&=1 + (3q + 4q^2 + 3q^3 + q^4)t+ (q^2 +  3q^3 + 4q^4 + 3q^5)t^2 + q^6t^3.
    \end{align*}
	
    \section{Statistics on integer compositions} \label{sec:results}
	The main results of this Section are consequences of Theorem~\ref{thm:macmahon}, which is attributed to MacMahon \cite{macmahon}. The presented proof follows the exposition of Knuth \cite[p.\ 17]{knuth}. But first, we state the following auxiliary result:
	\begin{lemma}\label{lmm:inv}
		Let $\sigma=(\sigma_1,\sigma_2,\ldots,\sigma_k)$ be a $k$-composition. Let $\pi=\pi_1\pi_2\ldots\pi_k$ be the unique permutation that sorts $\sigma$ in weakly decreasing order, $\sigma_{\pi_1}\geq\sigma_{\pi_2}\geq\cdots\geq\sigma_{\pi_k}$, with repeated entries sorted by increasing order of their indices, i.e., $\sigma_{\pi_j}=\sigma_{\pi_{j+1}}$ implies $\pi_j<\pi_{j+1}$. Then,
		\begin{subequations}
			\begin{align}
				\inv(\pi)&=\inv(\sigma^R),\label{eq:inva}\\
				\imaj(\pi)&=\comaj(\sigma^R),\label{eq:comaja}\\ 
				\icomaj(\pi)&=\maj(\sigma^R),\label{eq:comajb}\\
				\ides(\pi)&=\des(\sigma^R),\label{eq:comajc}
			\end{align}    
		\end{subequations}
		where $\sigma^R=(\sigma_k,\sigma_{k-1},\ldots,\sigma_1)$ is the reverse of $\sigma$.
	\end{lemma}
	\begin{proof}
        We prove that, for a given $d>0$, $\pi^{-1}_{i}>\pi^{-1}_{i+d}$ if and only if $\sigma^R_{k-i+1}<\sigma^R_{k-i-d+1}$. Recall that $\pi^{-1}_i$ is the position of $i$ in $\pi$. Since $\pi$ is the permutation that sorts $\sigma$ in weakly decreasing order, we have three possibilities: (1) $\sigma_i>\sigma_{i+d}$, then $i$ appears before $i+d$ in $\pi$, therefore $\pi^{-1}_i<\pi^{-1}_{i+d}$; (2) $\sigma_i=\sigma_{i+d}$, then $i$ appears before $i+d$ in $\pi$, due to the defined tie breaking rule, therefore $\pi^{-1}_i<\pi^{-1}_{i+d}$; (3) $\sigma_i<\sigma_{i+d}$, then $i$ appears after $i+d$ in $\pi$, therefore $\pi^{-1}_i>\pi^{-1}_{i+d}$. Hence, $\pi^{-1}_i>\pi^{-1}_{i+d}$ if and only if $\sigma_i<\sigma_{i+d}$, which is equivalent to $\sigma^R_{k-i+1}<\sigma^R_{k-i-d+1}$. If $d=1$, we conclude that $i$ is a descent of $\pi^{-1}$ if and only if $k-i$ is a descent of $\sigma^R$, proving \eqref{eq:comaja}, \eqref{eq:comajb}, and \eqref{eq:comajc}. If we let $i+d=j$, we conclude that $(i,j)$ is an inversion of $\pi^{-1}$ if and only if $(k-j+1,k-i+1)$ is an inversion of $\sigma^R$. Then, since $\inv(\pi^{-1})=\inv(\pi)$, Equation \eqref{eq:inva} follows.
	\end{proof}
	
	\begin{theorem}[MacMahon \cite{macmahon}]\label{thm:macmahon}
		Let $\sigma$ be a $k$-composition, let $\lambda$ be a $k$-partition, and let $\pi\in\mathbf{S}_k$. Then, there exists a bijective map $\sigma \mapsto (\pi,\lambda)$ such that $\pi$ is the unique permutation defined in Lemma~\ref{lmm:inv} and $\vert\lambda\vert+\maj(\pi)=\vert\sigma\vert.$
	\end{theorem}
	\begin{proof}
		Let $\sigma=(\sigma_1,\sigma_2,\ldots,\sigma_k)$ be a $k$-composition and let $\pi=\pi_1\pi_2\ldots\pi_k$ be the respective unique permutation defined in Lemma~\ref{lmm:inv}. We build a $k$-partition $\lambda$ as follows. First, let $\mu=(\mu_1,\mu_2,\ldots,\mu_k)=(\sigma_{\pi_1},\sigma_{\pi_2},\ldots,\sigma_{\pi_k})$. Then, for $1\leq j<k$, subtract from each $\mu_j$ the number of descents of $\pi$ appearing to the right of or at position $j$ in the one-line notation of $\pi$. We still have $\mu_1\geq \mu_2\geq\cdots\geq \mu_k\geq 1$ because $\mu_j$ is strictly greater than $\mu_{j+1}$ whenever $\pi_j>\pi_{j+1}$. After this procedure, the total reduction in $\mu$ is $\maj(\pi)$, resulting in $\mu'$, and taking $\lambda=\mu'$ yields $\vert\lambda\vert+\maj(\pi)=\vert\sigma\vert$. 
		Conversely, let $\lambda$ be a $k$-partition and $\pi$ be a permutation of $[\mathbf{k}]$. We can proceed backwards computing the integers $\mu_i$ from $\lambda$ and $\pi$, and then reconstruct the composition $\sigma$ by applying $\pi^{-1}$ to $\mu$. It is straightforward to see that the resulting $\sigma$ will be a $k$-composition satisfying $\vert\sigma\vert=\vert\lambda\vert+\maj(\pi)$.
	\end{proof}
	\begin{example}
		We apply the bijection described in the proof of Theorem~\ref{thm:macmahon} to the composition $\sigma=(4,2,1,2,1,5,3)$. In this case, $\vert\sigma\vert=18$, $\pi=6172435$, $D(\pi)=\{1,3,5\}$, $\maj(\pi)=9$, and $\mu=(5,4,3,2,2,1,1)$. Then, $\lambda=(2,2,1,1,1,1,1)$ and $\vert\lambda\vert=9$. Conversely, if we start with $\pi=6172435$ and $\lambda=(2,2,1,1,1,1,1)$, we get $\mu=(5,4,3,2,2,1,1)$. Applying $\pi^{-1}=2465713$ to $\mu$, we retrieve $\sigma=(4,2,1,2,1,5,3)$.
	\end{example}
	MacMahon used Theorem \ref{thm:macmahon} to prove that the distribution of $\maj$ over $\mathbf{S}_k$ is equal to $[k]_q!$, the same distribution of $\inv$, giving a bijective proof of the following generating function identity:
	\begin{align}\label{eq:genfuncid}
		\frac{q^k}{(1-q)^k}=[k]_q!\frac{q^k}{(q)_k}.
	\end{align}
	
    The right-hand side of \eqref{eq:genfuncid} is the generating function for the number of $k$-compositions of $n$, while the left-hand side is the product of the distribution of $\maj$ over $\mathbf{S}_k$ and the generating function for the number of $k$-partitions of $n$. 
    
    We are now ready to state the main results.
    \begin{theorem}\label{thm:jointstat}
		We have
		\begin{align}
\sum_{\sigma\in\mathbf{C}_{k}}p^{\vert\sigma\vert}q^{\inv(\sigma)}t^{\comaj(\sigma)}u^{\maj(\sigma)}v^{\des(\sigma)}=\frac{p^k}{(p)_k}\sum_{\pi\in\mathbf{S}_k}p^{\maj(\pi)}q^{\inv(\pi)}t^{\imaj(\pi)}u^{\icomaj(\pi)}v^{\ides(\pi)}.\label{eq:jointCk}
		\end{align}
	\end{theorem}
	\begin{proof}
		Let $\sigma$ be a $k$-composition of $n$, let $\pi\in\mathbf{S}_k$, and let $\lambda$ be a $k$-partition of $n$. Since the map $\sigma\mapsto\sigma^R$ is an involution on $\mathbf{C}_k$, the distribution of $(\vert\cdot\vert,\inv,\comaj,\maj,\des)$ over $\mathbf{C}_k$ is the same whether we sum over all $\sigma$ or over all $\sigma^R$. By Theorem~\ref{thm:macmahon} and Lemma~\ref{lmm:inv}, this distribution can be obtained by counting pairs $(\pi,\lambda)$ that have specified values of
        $\inv(\pi)$, $\imaj(\pi)$, $\icomaj(\pi)$, and $\ides(\pi)$, and, simultaneously, satisfy $\maj(\pi)+\vert\lambda\vert=n$. Thus, in terms of generating functions, Equation~\eqref{eq:jointCk} holds.
	\end{proof}
	It is straightforward to see that Theorem~\ref{thm:jointstat} generalizes \eqref{eq:genfuncid} and extends naturally to every composition statistic that can be expressed in terms of a corresponding statistic on the unique permutation defined in Lemma~\ref{lmm:inv}.
	\begin{corollary}\label{thm:inversions}
		Let $IC_q(p,k)=\sum_{n\geq 0}\sum_{r\geq 0}ic_r(n,k)p^nq^r$ and $IC_q(p)=\sum_{n\geq 0}\sum_{r\geq 0}ic_r(n)p^nq^r$ be the ordinary bivariate generating functions of $ic_r(n,k)$ and $ic_r(n)$, respectively. Then,
		\begin{align}
			IC_q(p,k)=\frac{p^k}{(p)_k}\sum_{\lambda\vdash k}f^{\lambda}(p)f^{\lambda}(q),\label{eq:icqpk}\\
			IC_q(p) = \sum_{\mu}\frac{p^{\vert\mu\vert}}{(p)_{\vert\mu\vert}}f^{\mu}(p)f^{\mu}(q),\label{eq:icqp}
		\end{align}
		where $\mu$ runs through all integer partitions. Additionally, the following recurrence holds:
		\begin{align}
			IC_q(p,k)=\frac{p^k}{1-p^k}\sum_{j=0}^{k-1}\gaussian{k}{j}_qIC_q(p,j),\label{eq:icqpk_rec}
		\end{align}
		with initial condition $IC_q(p,0)=1$.
	\end{corollary}
	\begin{proof}
		By Theorem~\ref{thm:jointstat},  
		\begin{align}
			IC_q(p,k)=\sum_{\sigma\in\mathbf{C}_{k}}p^{\vert\sigma\vert}q^{\inv(\sigma)}=\frac{p^k}{(p)_k}\sum_{\pi\in\mathbf{S}_k}p^{\maj(\pi)}q^{\inv(\pi)}=\frac{p^k}{(p)_k}H_k(p,q).\label{eq:aux}
		\end{align}
		Replacing \eqref{eq:hxy} in \eqref{eq:aux} yields \eqref{eq:icqpk}. Summing \eqref{eq:icqpk} over all possible $k$ results in \eqref{eq:icqp}. Equation~\eqref{eq:icqpk_rec} follows after trivial manipulations in \eqref{eq:carlitzRec}.
	\end{proof}     
    It is well-known that $(\vert\cdot\vert,\maj)$ and $(\vert\cdot\vert,\inv)$ have the same distribution over $\mathbf{C}_k$, which is, therefore, given by \eqref{eq:icqpk}. It is not difficult to see that $(\vert\cdot\vert,\comaj)$ also shares the same distribution. Therefore, from the list of composition statistics given in Lemma~\ref{lmm:inv}, it remains to analyze the distribution of $(\vert\cdot\vert,\des)$ over $\mathbf{C}_k$.
    
    \begin{corollary}\label{thm:descents}
        Let $DC_t(q,k)=\sum_{n\geq0}\sum_{r\geq0}dc_r(n,k)q^nt^r$ and $DC_t(q)=\sum_{n\geq0}\sum_{r\geq0}dc_r(n)q^nt^r$ be the ordinary bivariate generating functions of $dc_r(n,k)$ and $dc_r(n)$, respectively. Then, $DC_t(q,0)=1$ and, for $k\geq1$,
\begin{align}
    &DC_t(q,k)=\frac{q^k}{(q)_k}\sum_{\lambda\vdash k}t^{l(\lambda)-1}(1-t)^{k-l(\lambda)}g^\lambda(q),\label{eq:dcpk}\\
    &DC_t(q)=1+\sum_{\mu}\frac{q^{\vert\mu\vert}}{(q)_{\vert\mu\vert}}t^{l(\mu)-1}(1-t)^{\vert\mu\vert-l(\mu)}g^\mu(q) = \frac{1-t}{\sum_{j\geq0}q^{\binom{j+1}{2}}\frac{(t-1)^j}{(q)_j}-t},\label{eq:dcp}
\end{align}
where $\mu$ runs through all integer partitions, except the empty partition.
    \end{corollary}
    \begin{proof}
    By Theorem~\ref{thm:jointstat},
    \begin{align}
DC_t(q,k)=\sum_{\sigma\in\mathbf{C}_{k}}q^{\vert\sigma\vert}t^{\des(\sigma)}=\frac{q^k}{(q)_k}\sum_{\pi\in\mathbf{S}_k}q^{\maj(\pi)}t^{\ides(\pi)}.\label{eq:desc}
    \end{align}
    It is well-known that the Foata bijection \cite{fs} $\phi:\mathbf{S}_k\rightarrow\mathbf{S}_k$  satisfies, for all $\pi\in\mathbf{S}_k$, $\maj(\pi)=\inv(\phi(\pi))$ and $D(\pi^{-1})=D((\phi(\pi))^{-1})$. Therefore, Equation~\eqref{eq:desc} can be rewritten as
        \begin{align}
DC_t(q,k)=\frac{q^k}{(q)_k}\sum_{\pi\in\mathbf{S}_k}q^{\inv(\pi)}t^{\des(\pi)}=\frac{q^k}{(q)_k}A_k(q,t).\label{eq:stateuler}
    \end{align}
    Replacing \eqref{eq:qeuler} in \eqref{eq:stateuler} yields \eqref{eq:dcpk}. Summing \eqref{eq:dcpk} over all possible $k$ is equivalent to set $z=\frac{q}{1-q}$ in \eqref{eq:geneuler}, resulting in \eqref{eq:dcp}.
    \end{proof}
    The last expression in Equation~\eqref{eq:dcp} is not new, as it is stated in a more general form by Heubach and Mansour \cite[Thm.\ 4.3]{heubach2}.
    
To compute the coefficients of $p^nq^r$ in the generating function \eqref{eq:icqp}, for instance, one must sum the expression in \eqref{eq:icqpk} over all partitions $\lambda$ satisfying $\vert\lambda\vert\leq n$, and then extract the desired coefficient. Obviously, after this procedure, the coefficients of $p^iq^j$ for all $j$ and for all $i<n$ will also be correct. Since the number of partitions of $n$ grows exponentially, this computation becomes increasingly difficult as $n$ increases, but it is still more efficient than running through all $2^{n-1}$ compositions of $n$ and computing the distribution of $\inv$ directly. The same reasoning applies to \eqref{eq:dcp}. 

Finally, we list some values of $ic_r(n)$ and $dc_r(n)$ in Tables~\ref{tb:icrn} and \ref{tb:dcrn}, respectively.

\begin{table}[H]
\begin{center}
\begin{tabular}{c|rrrrrrrrrrrrr}
\hline
$n$/$r$ & 0 & 1 & 2 & 3 & 4 & 5 & 6 & 7 & 8 & 9 & 10 & 11 & 12\\
\hline
0 & 1 & 0 & 0 & 0 & 0 & 0 & 0 & 0 & 0 & 0 & 0 & 0 & 0 \\
1 & 1 & 0 & 0 & 0 & 0 & 0 & 0 & 0 & 0 & 0 & 0 & 0 & 0 \\
2 & 2 & 0 & 0 & 0 & 0 & 0 & 0 & 0 & 0 & 0 & 0 & 0 & 0 \\
3 & 3 & 1 & 0 & 0 & 0 & 0 & 0 & 0 & 0 & 0 & 0 & 0 & 0 \\
4 & 5 & 2 & 1 & 0 & 0 & 0 & 0 & 0 & 0 & 0 & 0 & 0 & 0 \\
5 & 7 & 5 & 3 & 1 & 0 & 0 & 0 & 0 & 0 & 0 & 0 & 0 & 0 \\
6 & 11 & 8 & 7 & 4 & 2 & 0 & 0 & 0 & 0 & 0 & 0 & 0 & 0 \\
7 & 15 & 15 & 14 & 10 & 6 & 3 & 1 & 0 & 0 & 0 & 0 & 0 & 0 \\
8 & 22 & 23 & 26 & 21 & 17 & 10 & 6 & 2 & 1 & 0 & 0 & 0 & 0 \\
9 & 30 & 37 & 44 & 42 & 36 & 27 & 19 & 11 & 6 & 3 & 1 & 0 & 0 \\
10 & 42 & 55 & 73 & 74 & 73 & 60 & 50 & 34 & 24 & 13 & 8 & 4 & 2 \\
11 & 56 & 83 & 115 & 128 & 133 & 123 & 109 & 87 & 68 & 48 & 32 & 20 & 12 \\
12 & 77 & 118 & 177 & 209 & 235 & 230 & 223 & 192 & 166 & 129 & 100 & 70 & 51 \\
13 & 101 & 171 & 265 & 333 & 391 & 412 & 419 & 392 & 359 & 308 & 256 & 203 & 157 \\
14 & 135 & 238 & 391 & 512 & 636 & 700 & 754 & 743 & 724 & 657 & 589 & 499 & 420 \\
15 & 176 & 332 & 563 & 777 & 997 & 1156 & 1292 & 1343 & 1363 & 1315 & 1235 & 1116 & 990\\
16 & 231 & 453 & 803 & 1146 & 1536 & 1844 & 2148 & 2322 & 2461 & 2470 & 2435 & 2301 & 2148\\
\hline
\end{tabular}
\end{center}
\caption{\label{tb:icrn} Values of $ic_r(n)$ (\seqnum{A189074}), $0\leq n\leq 16$, $0\leq r\leq 12$.}
\end{table}
\begin{table}[H]
\begin{center}
\begin{tabular}{c|rrrrrr}
\hline
$n$/$r$ & 0 & 1 & 2 & 3 & 4 & 5 \\
\hline
0 & 1 & 0 & 0 & 0 & 0 & 0 \\
1 & 1 & 0 & 0 & 0 & 0 & 0 \\
2 & 2 & 0 & 0 & 0 & 0 & 0 \\
3 & 3 & 1 & 0 & 0 & 0 & 0 \\
4 & 5 & 3 & 0 & 0 & 0 & 0 \\
5 & 7 & 9 & 0 & 0 & 0 & 0 \\
6 & 11 & 19 & 2 & 0 & 0 & 0 \\
7 & 15 & 41 & 8 & 0 & 0 & 0 \\
6 & 22 & 77 & 29 & 0 & 0 & 0 \\
9 & 30 & 142 & 81 & 3 & 0 & 0 \\
10 & 42 & 247 & 205 & 18 & 0 & 0 \\
11 & 56 & 421 & 469 & 78 & 0 & 0 \\
12 & 77 & 689 & 1013 & 264 & 5 & 0 \\
13 & 101 & 1113 & 2059 & 786 & 37 & 0 \\
14 & 135 & 1750 & 4021 & 2097 & 189 & 0 \\
15 & 176 & 2712 & 7558 & 5179 & 751 & 8 \\
16 & 231 & 4128 & 13780 & 11998 & 2558 & 73\\
\hline
\end{tabular}
\end{center}
\caption{\label{tb:dcrn} Values of $dc_r(n)$ (\seqnum{A238343}, \seqnum{A238344}), $0\leq n\leq 16$, $0\leq r\leq 5$.}
\end{table}
\begin{remark}
Among the several other distributions of joint statistics on $\mathbf{C}_k$ that follow from Theorem~\ref{thm:jointstat}, we highlight the distribution of $(\vert\cdot\vert,\comaj,\des)$, denoted by $B_k(p,q,t)$. It is a known result, due to Gessel \cite[Thm.\ 8.4]{gessel}, that
\begin{align}
    A_k(p,q,t)\coloneqq\sum_{\pi\in\mathbf{S}_k}p^{\inv(\pi)}q^{\maj(\pi)}t^{\des(\pi)}\label{eq:akpqtdef}
\end{align}
has the following generating function:
\begin{align}
    \sum_{k\geq0}\frac{A_k(p,q,t)u^k}{(t;q)_{k+1}[k]_p!}=\sum_{k\geq0}t^k\qexp_p(u)\qexp_p(qu)\cdots \qexp_p(q^ku),\label{eq:akpqt}
\end{align}
where $(t;q)_{k+1}=(1-t)(1-tq)\cdots(1-tq^k)$ and $\qexp_p(u)=\sum_{k\geq 0}\frac{u^k}{[k]_p!}$. Now, if we set $\pi\rightarrow\pi^{-1}$ and use inverse Foata bijection \cite{fs}, recalling that it preserves the descent set of $\pi^{-1}$, we can rewrite \eqref{eq:akpqtdef} as
\begin{align}
    A_k(p,q,t)=\sum_{\pi\in\mathbf{S}_k}p^{\maj(\pi)}q^{\imaj(\pi)}t^{\ides(\pi)}.
\end{align}
Then, by Theorem~\ref{thm:jointstat},
\begin{align*}
B_k(p,q,t)=\sum_{\sigma\in\mathbf{C}_k}p^{\vert\sigma\vert}q^{\comaj(\sigma)}t^{\des(\sigma)}=\frac{p^k}{(p)_k}A_k(p,q,t).
\end{align*}
Setting $u=\frac{p}{1-p}$ in \eqref{eq:akpqt} yields an expression involving $B_k(p,q,t)$:
\begin{align}
    \sum_{k\geq0}\frac{B_k(p,q,t)}{(t;q)_{k+1}}=\sum_{k\geq0}t^k\qexp_p\left(\frac{p}{1-p}\right)\qexp_p\left(\frac{qp}{1-p}\right)\cdots \qexp_p\left(\frac{q^kp}{1-p}\right).
\end{align}
\end{remark}

	\section{Acknowledgments}
	I am grateful to God Almighty, for His mercies never end, they are new every morning, and great is His faithfulness.

	\bigskip
	\hrule
	\bigskip
	
	\noindent 2020 {\it Mathematics Subject Classification}:
	Primary 05A15; Secondary 05A05, 05A17.
	
	\noindent \emph{Keywords: } inversion, descent, integer composition.
	
	\bigskip
	\hrule
	\bigskip
	
	\noindent (Concerned with sequences	\seqnum{A189052},
	\seqnum{A189073},
	\seqnum{A189074},
	\seqnum{A238343},
	\seqnum{A238344}.)
	
	\bigskip
	\hrule
	\bigskip
	
	\vspace*{+.1in}
	\noindent
	Received XXXXXX XX 20XX;
	revised versions received XXXXXX XX 20XX; XXXXXX XX 20XX.
	Published in {\it Journal of Integer Sequences}, XXXXXX XX 20XX.
	
	\bigskip
	\hrule
	\bigskip
	
	\noindent
	Return to
	\href{https://cs.uwaterloo.ca/journals/JIS/}{Journal of Integer Sequences home page}.
	\vskip .1in
	
\end{document}